\documentclass{article}

\usepackage{xparse}
\usepackage{etoolbox}
\usepackage{kvoptions}

\usepackage{fix-cm}
\usepackage[T1]{fontenc}
\usepackage[utf8]{inputenc}
\usepackage[full]{textcomp}
\usepackage[l2tabu, orthodox]{nag}

\usepackage{lmodern} 
\usepackage[bitstream-charter,cal=cmcal]{mathdesign}

\usepackage{microtype}
\usepackage{csquotes}

\usepackage{graphicx}
\usepackage{tikz}
\usetikzlibrary{matrix, arrows, decorations.markings, decorations.pathmorphing, calc}

\tikzset{->-/.style={decoration={
  markings,
  mark=at position 0.5 with {\arrow{stealth}}},postaction={decorate}}}
\tikzset{->>-/.style={decoration={
  markings,
  mark=at position 0.5 with {\arrow{>>}}},postaction={decorate}}}
\tikzset{snake it/.style={decorate, decoration=snake}}

\usepackage{tikz-cd}
\usepackage{wrapfig}

\usepackage{caption}
\usepackage{subcaption}
\usepackage{multirow}
\usepackage{booktabs}
\usepackage{tabularx}
\usepackage{tabulary}
\usepackage{longtable}
\usepackage{siunitx}
\usepackage{pdflscape}
\usepackage{rotating}

\newcolumntype{R}{>{$}r<{$}}
\newcolumntype{C}{>{$}c<{$}}
\newcolumntype{L}{>{$}l<{$}}

\graphicspath{{./PointCountDiagrams/}}

\usepackage{dsfont}
\usepackage{xfrac}
\usepackage{mathtools}
\usepackage{mleftright}
\usepackage{manfnt}
\usepackage{accents}

\usepackage{emptypage}
\usepackage{comment}

\usepackage[inline]{enumitem}
\usepackage[obeyFinal,colorinlistoftodos=true,textsize=footnotesize]{todonotes}

\usepackage[hyphens]{url}

\usepackage{hyperref}
\usepackage{geometry} 
\usepackage{amsthm}
\usepackage{thmtools,thm-restate}
\usepackage[capitalize]{cleveref}

\definecolor{dark-red}{rgb}{0.4,0.15,0.15}
\definecolor{dark-blue}{rgb}{0.15,0.15,0.4}
\definecolor{medium-blue}{rgb}{0,0,0.5}
\hypersetup{
    colorlinks, linkcolor={dark-red},
    citecolor={dark-blue}, urlcolor={medium-blue}
}

\LetLtxMacro{\amsmathdots}{\dots}
\usepackage{ellipsis}
\AtBeginDocument{\LetLtxMacro{\dots}{\amsmathdots}}

\DeclareMathOperator{\ord}{ord}

\DeclareMathOperator{\Aut}{Aut}

\DeclareMathOperator{\Sym}{Sym}
\DeclareMathOperator{\Stab}{Stab}

\DeclareMathOperator{\PGL}{PGL}

\DeclareMathOperator{\PConf}{PConf}
\DeclareMathOperator{\UConf}{UConf}
\DeclareMathOperator{\Frob}{Frob}

\DeclareMathOperator{\Tr}{Tr}

\newcommand*{\et}{\text{ét}}

\newcommand*{\dispunct}[1]{\,\text{#1}}

\newcommand{\from}{\vcentcolon}

\newcommand{\divides}{\mathbin{\vert}}

\newcommand{\dsum}{\oplus}
\newcommand{\dSum}{\bigoplus}
\newcommand{\tensor}{\otimes}

\renewcommand{\emptyset}{\varnothing}

\NewDocumentCommand\xDeclarePairedDelimiter{mmm}
 {%
  \NewDocumentCommand#1{som}{%
   \IfNoValueTF{##2}
    {\IfBooleanTF{##1}{#2##3#3}{\mleft#2##3\mright#3}}
    {\mathopen{##2#2}##3\mathclose{##2#3}}%
  }%
 }

\xDeclarePairedDelimiter{\abs}{\lvert}{\rvert}
\xDeclarePairedDelimiter{\norm}{\lVert}{\rVert}
\xDeclarePairedDelimiter{\floor}{\lfloor}{\rfloor}
\xDeclarePairedDelimiter{\ceil}{\lceil}{\rceil}
\xDeclarePairedDelimiter{\gen}{\langle}{\rangle}
\xDeclarePairedDelimiter{\pseries}{\llbracket}{\rrbracket}
\xDeclarePairedDelimiter{\oneto}{[}{]}
\xDeclarePairedDelimiter{\parenth}{(}{)}

\NewDocumentCommand{\set}{somm}{%
   \IfNoValueTF{#2}
    {\IfBooleanTF{#1}{\{#3 \mid #4\}}{\mleft\{ #3 \mathrel{}\middle\vert\mathrel{} #4 \mright\}}}
    {\mathopen{#2\{}#3 \mathrel{}#2\vert\mathrel{} #4\mathclose{#2\}}}%
  }
\NewDocumentCommand{\present}{somm}{%
   \IfNoValueTF{#2}
    {\IfBooleanTF{#1}{\langle#3 \mid #4\rangle}{\mleft\langle#3 \mathrel{}\middle\vert\mathrel{} #4 \mright\rangle}}
    {\mathopen{#2\langle}#3 \mathrel{}#2\vert\mathrel{} #4\mathclose{#2\rangle}}%
  }
\NewDocumentCommand{\inner}{somm}{%
   \IfNoValueTF{#2}
    {\IfBooleanTF{#1}{\langle#3 , #4\rangle}{\mleft\langle#3 , #4 \mright\rangle}}
    {\mathopen{#2\langle}#3 , #4\mathclose{#2\rangle}}%
  }

\newcommand{\CC}{\mathbb{C}}
\newcommand{\FF}{\mathbb{F}}

\newcommand{\PP}{\mathbb{P}}
\newcommand{\QQ}{\mathbb{Q}}

\newcommand{\ZZ}{\mathbb{Z}}

\newcommand{\fS}{\mathfrak{S}}

\newcommand*{\widebar}[1]{\mkern 1.5mu\overline{\mkern-1.5mu#1\mkern-1.5mu}\mkern 1.5mu}

\newcommand*{\cl}[1]{
\begingroup
    \setbox\z@=\hbox{\ensuremath{#1}}%
    \ifdimgreater{\wd\z@}{4em}{\mleft(#1\mright)^{-}}{\widebar{#1}}
\endgroup
}

\newcommand*{\interior}[1]{
\begingroup
    \setbox\z@=\hbox{\ensuremath{#1}}%
    \ifdimgreater{\wd\z@}{1.5em}{\mleft(#1\mright)^{\circ}}{\accentset{\circ}{#1}}
\endgroup
}

\numberwithin{equation}{section}
\declaretheorem[sibling=equation]{theorem}

\declaretheorem[sibling=theorem]{corollary}
\declaretheorem[sibling=theorem]{proposition}

\declaretheorem[numbered=no, title=Theorem]{theorem*}
\declaretheorem[numbered=no, title=Corollary]{corollary*}
\declaretheorem[numbered=no, title=Lemma]{lemma*}
\declaretheorem[numbered=no, title=Proposition]{proposition*}
\declaretheorem[numbered=no, title=Conjecture]{conjecture*}

\declaretheorem[numbered=no, style=definition, title=Definition]{definition*}

\declaretheorem[numbered=no, style=definition, title=Exercise]{exercise*}

\declaretheorem[sibling=theorem, style=remark]{remark}
\declaretheorem[numbered=no, style=remark, title=Remark]{remark*}

\declaretheorem[numbered=no, style=remark, title=Example]{example*}

\usepackage{sseq}

\usepackage[style=ieee-alphabetic,url=false]{biblatex}

\renewbibmacro*{doi+eprint+url}{%
  \iftoggle{bbx:eprint}
    {\usebibmacro{eprint}}
    {}%
  \newunit\newblock
  \iftoggle{bbx:url}
    {\usebibmacro{url+urldate}}
    {}}

\newcommand*{\Gr}{\mathrm{Gr}}

\newlist{singularity}{enumerate}{2}
\setlist[singularity,1]{label=(\Roman*),noitemsep, ref=\Roman*}
\setlist[singularity,2]{label=(\alph*),noitemsep, ref=\alph*}

\makeatletter
\def\paragraph{\@startsection{paragraph}{4}%
  \z@\z@{-\fontdimen2\font}%
  {\normalfont\bfseries}}
\makeatother

\title{Arithmetic statistics on cubic surfaces}
\author{Ronno Das}

\begin{document}

\maketitle

\begin{abstract}
In this paper we compute the distributions of various markings on smooth cubic surfaces defined over the finite field $\FF_q$, for example the distribution of pairs of points, `tritangents' or `double sixes'.
We also compute the (rational) cohomology of certain associated bundles and covers over complex numbers.
\end{abstract}

\section{Introduction}

The classical Cayley--Salmon theorem implies that each smooth cubic surface over an algebraically closed field contains exactly $27$ lines (see \cref{definitions} for detailed definitions).
In contrast, for a surface over a finite field $\FF_q$, all $27$ lines are defined over $\overline{\FF}_q$ but not necessarily over $\FF_q$ itself.
In other words, the action of the Frobenius $\Frob_q$ permutes the $27$ lines and only fixes a (possibly empty) subset of them.
It is also classical that the group of all such permutations, which can be identified with the Galois group of an appropriate extension or cover, is isomorphic to the Weyl group $W(E_6)$ of type $E_6$.

This permutation of the $27$ lines governs much of the arithmetic of the surface $S$: evidently the pattern of lines defined over $\FF_q$ and, less obviously, the number of $\FF_q$ points on $S$ (or $\UConf^n S$ etc).
Work of Bergvall and Gounelas \cite{BG2019} allows us to compute the number of cubic surfaces over $\FF_q$ where $\Frob_q$ induces a given permutation, or rather a permutation in a given conjugacy class of $W(E_6)$.
The results in this paper can be thought of as a combinatorial (\cref{point-count-theorem}) or representation-theoretic (\cref{cohomology-theorem}) reinterpretation of their computation.

\begin{theorem}\label{point-count-theorem}
Over the finite field $\FF_q$, the number of smooth cubic surfaces on whose $27$ lines $\Frob_q$ acts by a given conjugacy class of $W(E_6)$ is as in \cref{point-count-by-conjugacy-class}.
\end{theorem}

\begin{table}
\caption{The number of cubic surfaces over $\FF_q$ on whose $27$ lines $\Frob_q$ acts by a given conjugacy class of $W(E_6)$. The factors in the second column are to normalize to a degree $4$ monic polynomial in $q$ (and come up naturally in the representation theoretic setup, see \cref{twisted-point-counts}). Some of these factors appear in later tables for the same reasons. The third column lists $q$ for which the count in the second column vanishes. See \cref{conjugacy-class-notation} for the notation used for conjugacy classes.}
\label{point-count-by-conjugacy-class}
\centering
\begin{tabular}{cRc}
\toprule
Conjugacy class $c$ & \dfrac{\#\{S| \Frob_{q,S} \sim c\}}{\#\PGL(4,\FF_q)} \times \dfrac{\# W(E_6)}{\# c}  & $\#\{S| \Frob_{q,S} \sim c\} = \emptyset$ for $q = $\\ \midrule
            $(1^6)$             &    (q-2)(q-3)(q-5)^2 & $2, 3, 5$ \\
          $(1^2,2^2)$           &    (q+1)^2(q-2)(q-3) &  $2, 3$   \\
        $(1^{-2},2^4)$          & (q-2)(q-3)(q^2-2q-7) &  $2, 3$   \\
           $(1^3,3)$            &      q(q+1)(q^2-q+1) &           \\
        $(1^{-3},3^3)$          &     (q+1)^2(q^2+q-3) &           \\
            $(3^2)$             &  (q-2)(q^3-q^2-2q-6) &    $2$    \\
      $(1^2,2^{-2},4^2)$        &         (q+1)^3(q-2) &    $2$    \\
            $(2,4)$             &    (q+1)(q-2)(q^2+1) &    $2$    \\
            $(1,5)$             &           q^2(q^2+1) &           \\
       $(1,2,3^{-1},6)$         &      q(q+1)(q^2+q-1) &           \\
       $(1^{-1},2^2,3)$         &      q(q+1)(q^2-q+1) &           \\
        $(1^{-2},2,6)$          &      q(q-2)(q^2+q+2) &    $2$    \\
    $(1,2^{-2},3^{-1},6^2)$     & (q+1)(q^3-2q^2+2q-3) &           \\
         $(3^{-1},9)$           &      q(q+1)(q^2-q+1) &           \\
$(1^{-1},2,3,4^{-1},6^{-1},12)$ &     (q+1)^2(q^2-q+1) &           \\ \midrule

           $(1^4,2)$            &     q(q-1)(q^2-4q+5) &           \\
            $(2^3)$             &        q(q-1)(q^2-3) &           \\
           $(1^2,4)$            &        q(q+1)^2(q-1) &           \\
       $(1^{-2},2^2,4)$         &             q(q-1)^3 &           \\
           $(1,2,3)$            &      q(q-1)(q^2-q-1) &           \\
       $(1^{-2},2,3^2)$         &     q(q-1)(q^2+2q+2) &           \\
             $(6)$              &             q^3(q-1) &           \\
        $(2,4^{-1},8)$          &        q(q+1)(q^2+1) &           \\
        $(1^{-1},2,5)$          &        q^2(q+1)(q-1) &           \\
    $(1,2^{-1},3^{-1},4,6)$     &      q(q-1)(q^2+q+1) &           \\ \bottomrule
\end{tabular}
\end{table}

It is worth noting how this relates to the distribution predicted by the Cebotarev density theorem: for a fixed smooth cubic surface defined over $\ZZ$, the conjugacy class of $\Frob_p$ acting on the $27$ lines of the $\bmod\ p$ reduction is distributed (as $p \to \infty$) proportional to the sizes of the conjugacy classes.
\Cref{point-count-theorem}, on the other hand, describes for each fixed $q$ the distribution over all smooth cubic surfaces defined over $\FF_q$.
The asymptotic distribution as $q \to \infty$ is still proportional to the size of the conjugacy class (with the normalization factor in \cref{point-count-by-conjugacy-class}, this corresponds to the polynomials listed each being monic); this is a relatively easy application of the Grothendieck--Lefschetz trace formula for twisted coefficients (see \cref{twisted-point-counts}).

It also follows from the trace formula that a cubic surface over $\FF_q$ contains $q^2 + tq + 1$ points, where $t$ is the trace of $\Frob_q$ on an appropriate representation of $W(E_6)$ (specifically, $V_1 \dsum V_6$ in the notation of \cref{irrep-notation}).
Adding up the counts for a given $t$, we get the number of surfaces containing a given number of points.

\begin{corollary}\label{point-count-by-point-count-corollary}
Over the finite field $\FF_q$, the number of smooth cubic surfaces with $q^2 + tq + 1$ points is given by \cref{point-count-by-point-count}.
\end{corollary}

\begin{table}
\caption{The number of cubic surfaces over $\FF_q$ with a given number of points.}
\label{point-count-by-point-count}
\centering
\begin{tabular}{RR}
	\toprule
	 t & \#\{S(\FF_q) = q^2 + tq+1\} \times \dfrac{\# W(E_6)}{\#\PGL(4,\FF_q)} \\ \midrule
	-2 &                                                    80(q^2+q-3)(q+1)^2 \\
	-1 &                            45 (77 q^4 - 43 q^3 + 45 q^2 - 181 q - 42) \\
	 0 &                                           432(27q^3-17q^2+5q+10)(q+1) \\
	 1 &                           60 (347 q^4 - 51 q^3 + 27 q^2 + 161 q - 12) \\
	 2 &                             144 (91 q^4 - 5 q^3 + 36 q^2 - 35 q - 15) \\
	 3 &                                                270(9q^2-13q+2)(q+1)^2 \\
	 4 &                                                    240(q^2-q+1)(q+1)q \\
	 5 &                                                    36(q^2-4q+5)(q-1)q \\
	 7 &                                                     (q-2)(q-3)(q-5)^2 \\ \bottomrule
\end{tabular}
\end{table}

In particular, a cubic surface for a pair $(t,q)$ exists unless the polynomial listed in \cref{point-count-by-point-count} vanishes at $q$ (i.e.\ unless $t = 7$ and $q = 2, 3, 5$).
These pairs were known before, by \cite{SwinnertonDyer2010} and \cite{BFL2018}, but the exact number of surfaces were not.
We can also do this for each conjugacy class individually (as in the third column of \cref{point-count-by-conjugacy-class}) and recover more recent results of Loughran and Trepalin \cite{LT2018}.

\begin{corollary}[{\cite[Theorem~1.1]{LT2018}}]
Other than the following exceptions, over every prime power $q$, every conjugacy class of $W(E_6)$ occurs as the class of $\Frob_q$ acting on the lines of some cubic surface:
\begin{itemize}[noitemsep]
\item the conjugacy class $(1^6)$ (i.e. identity) does not appear for $q = 2, 3, 5$;
\item the conjugacy classes $(1^2,2^2)$ and $(1^{-2},2^4)$ do not appear for $q = 2, 3$;
\item the conjugacy classes $(3^2)$, $(1^{-2},2^2,4)$, $(2,4)$ and $(1^{-2},2,6)$ do not appear for $q = 2$.
\end{itemize}
\end{corollary}

Just like the number of lines, the intersection pattern of pairs of lines is fixed over $\overline{\FF}_q$.
Once we know how $\Frob_{q}$ acts on the set of lines on $S$, we can determine how many sets of lines with a given intersection pattern is fixed by $\Frob_{q}$ (and hence is defined over $\FF_q$).
Since we know the distribution of conjugacy classes of $\Frob_q$, this allows us to compute the distribution of that intersection pattern over all $S$ defined over $\FF_q$.
Similarly, we can also find the distribution of number of $\FF_q$ points on $\Sym^n S$ or $\UConf^n S$.
Two such examples are listed in \cref{tritangents-table} (tritangents), and \cref{pairs-table} (unordered pairs of points).

The paper of Bergvall--Gounelas computes the $\fS_6$-equivariant point counts of smooth cubic surfaces \cite[Table~1]{BG2019}, i.e. the analogue of \cref{point-count-by-conjugacy-class} for $\fS_6 < W(E_6)$.
This determines $H^*_{\et}(X(\overline{\FF}_q; \QQ_\ell)$ (see \cref{definitions,cohomology-of-Y}) as $\fS_6$ representations and leaves finitely many possibilities as $W(E_6)$ representations.
They then use various constraints on the $W(E_6)$-equivariant point counts (for example that each entry in \cref{point-count-by-conjugacy-class} must be non-negative for each $q$) to reduce to a single possibility.

Our paper fully explores the point-count implications of this computation.
Further, we use their computation of cohomology, along with some facts from \cite{Das2019} to compute the singular cohomology of the moduli space of smooth cubic surfaces with certain marked patterns of lines and points as $W(E_6)$ representations.
For the precise statement, see \cref{cohomology-theorem} and the preceding definitions.

\subsection{Acknowledgements}
I would like to thank Benson Farb for suggesting the project, for guidance and numerous helpful comments throughout the composition of this document.
Thanks to Nir Gadish for pointing out an error in an earlier version of \cref{cohomology-theorem}.
I must also thank Nir, and Ben O'Connor, for helping me understand the Grothendieck--Lefschetz trace formula for twisted coefficients.
Thanks to Nate Harman and Nat Mayer for generally helpful conversations.
I was supported by the Jump Trading Mathlab Research Fund.

\section{Marking points and lines on cubic surfaces}
Denote the \emph{ordered configuration space} of $n$ points on a space $X$ by
\[\PConf^n X \vcentcolon= \set{(x_1, \dots, x_n) \in X^n}{x_i \ne x_j \text{ for } i \ne j} \dispunct.\]
The \emph{unordered configuration space} of $n$ points on $X$ is the quotient by the action of the symmetric group $\fS_n$ permuting the $n$ points, and will be denoted by
\[\UConf^n X \vcentcolon= \PConf^n X/\fS_n \dispunct.\]

\label{definitions}
A cubic surface $S \subset \PP^3$ is defined by a homogeneous cubic polynomial $F$ in $4$ variables.
Accordingly, the space of cubic surfaces is the projectivization $\PP^{19}$ of the affine space of homogeneous cubic polynomials in $4$ variables.
The cubic surface $S$ is \emph{smooth} if its defining polynomial $F$ is smooth, i.e. its partials do not simultaneously vanish on $S$.

The subspace of non-smooth or \emph{singular} cubic surfaces is a closed set $\Sigma$, given by a `discriminant' polynomial (with integer coefficients) in the coefficients of $F$.
We will denote the complement, the space of smooth cubic surfaces, by
\[X = \PP^{19} \setminus \Sigma \dispunct.\]
Thus $X$ is an algebraic variety defined over $\ZZ$.

According to the classical theorem of Cayley and Salmon, a smooth cubic surface over $\CC$ contains $27$ lines.
This corresponds to a degree $27$ cover of $X(\CC)$ given by the incidence variety of lines $\ell \in \Gr(2,4)$ (the Grassmannian of lines in $\PP^3$) and cubic surfaces $S \in X$ (see \cite{Das2018}).
This cover is not Galois; its Galois group is $W(E_6)$, the Weyl group of type $E_6$ (see \cite{Manin1986,Harris1979}).
More precisely, this is the subgroup of $\fS_{27}$ permuting the lines that preserves the set of intersecting pairs of lines.
Explicitly, let $R \subset \{1,\dots,27\}^2$ be the set of intersecting pairs of lines for some fixed $S_0 \in X$.
Then the Galois cover is
\[Y = \set{(S, L_1, \dots, L_{27})}{L_i \subset S, \; L_i \cap L_j \ne \emptyset \text{ iff } (i,j) \in R} \subset X \times \PConf^{27}\Gr(2,4) \dispunct.\]
Note that $Y$ is an algebraic variety over $X$ and, by above, $Y(\CC) \to X(\CC)$ is a Galois cover with Galois group $W(E_6)\cong \Stab(R) < \fS_{27}$.
This finite subgroup is of order $51840$ and has a simple subgroup of index $2$; we describe some of its representation theory in \cref{character-table-section}.

There are also actions of $\PGL(4) = \Aut(\PP^3)$ on $X$ and on $Y$ such that the covering map is equivariant.
Both the actions over $\CC$ have closed orbits and finite stabilizers, which are subgroups of $W(E_6)$.

In the following, the summands on the right denote irreducible representations of $W(E_6)$, for details and the notation see \cref{irrep-notation,character-table}.
In particular, $V_d$ or $U_d$ has dimension $d$ and $V_1$ is the trivial representation.

\begin{theorem}[{\cite[Theorem~1.2]{BG2019}}]
\label{cohomology-of-Y}
Let $H^i = H^i(Y(\CC)/\PGL(4,\CC); \QQ)$.
Then as representations of $W(E_6)$:
\begin{align*}
H^0 &= V_1 \dispunct;\\
H^1 &= V_{15,2}\dispunct;\\
H^2 &= V_{81}\dispunct;\\
H^3 &= V_{15,1} \dsum U_{80} \dsum U_{90}\dispunct;\\
H^4 &= V_{30} \dsum V_{30}' \dsum U_{10} \dsum U_{80} \dispunct.
\end{align*}
For dimensional reasons, $H^i = 0$ for other $i$.
\end{theorem}

This is sufficient to determine $H^*(Y(\CC); \QQ)$ as a $W(E_6)$ representation since it follows from the work of Peters and Steenbrink \cite{PS2003} that
\[H^*(Y(\CC); \QQ) \cong H^*(Y(\CC)/\PGL(4,\CC); \QQ) \tensor H^*(\PGL(4, \CC); \QQ) \dispunct,\]
where the action on $\PGL(4,\CC)$ (and its cohomology) is trivial.

A subgroup $G < W(E_6)$ corresponds to an intermediate cover $Y/G \to X$.
These often also correspond to marking a (labeled) subset of lines on $S \in X$, when $G$ is the (pointwise) stabilizer of such a set.
For example, marking one line $L$ corresponds to $\Stab(L) \cong W(D_5)$ (see \cite{Das2018}), and marking a `tritangent' $T$ corresponds to $\Stab(T) \cong W(F_4)$ (see \cite{Naruki1982}).
Note that $Y = Y/\{1\}$ and $X = Y/W(E_6)$ are the `trivial' examples.

We denote the \emph{incidence variety} of points in $\PP^3$ and $S \in X$ by
\[U \vcentcolon= \set{(S,p)}{p \in S} \subset X \times \PP^3 \dispunct.\]
Then $U \to X$ is the `universal family' of smooth cubic surfaces and $U(\CC)$ is a fiber bundle over $X(\CC)$ with fiber $S \subset \PP^3(\CC)$ over $S \in X(\CC)$ (see \cite{Das2019}).
We can also construct various associated spaces over $X$: $\Pi^n_X U$ with fiber $S^n$, $\Sym^n_X U$ with fiber $\Sym^n S$, $\PConf^n_X U$ with fiber $\PConf^n S$, etc.

Finally, we can combine the two constructions above by taking more fiber products over $X$, for instance:
\[(Y/W(D_5)) \times_X (\Pi^2_X U) = \set{(S, L, p_1, p_2)}{L \subset S, p_i \in S} \subset X \times \Gr(2,4) \times (\PP^3)^2 \dispunct.\]
Of course,
\[(Y/G) \times_X U = (Y \times_X U)/G \dispunct,\]
where the action of $G \subset W(E_6)$ on $U$ is trivial, and similarly for $(Y/G) \times_X (\Pi^n_X U)$ etc.
It is worth noting that in these examples there is no enforced relation between the points and the lines marked, in particular we do not insist that $p_i \in L$.

\begin{remark}
We should not expect the results and techniques of this paper to apply if we do insist e.g.\ that $p_i \in L$.
For one, the space
\[\set{(S, L, p)}{p \in L \subset S} \subset (X \times \Gr(2,4) \times \PP^3)(\CC)\]
is \emph{not} a fiber bundle over $X(\CC)$ (due to the existence of \emph{Eckardt points}, i.e. triple intersections of lines, on some special $S \in X$).
\end{remark}

The bundles $U \to X$ and the associated constructions each have monodromy: $H^*(S)$ is a $\pi_1(X,S)$ representation.
However, the monodromy representation factors through $W(E_6)$ (see \cref{surface-cohomology} for an explicit description of the representation) and consequently in the pullback bundle $Y \times_X U \to Y$ the monodromy action of $\pi_1(Y)$ on $H^*(S)$ is trivial.
The same holds for $Y \times_X (\Pi^n_X U)$ etc.

\begin{theorem} \label{cohomology-theorem}
Let $Z$ be $\Pi^n_X U$ or $\Sym^n_X U$.
Let $F$ be the fiber of $Z(\CC) \to X(\CC)$ over $S$ (i.e.\ $S^n$ or $\Sym^n S$ respectively).
Then
\[H^*((Y \times_X Z)(\CC); \QQ) \cong H^*(Y(\CC); \QQ) \tensor H^*(F; \QQ)\]
as both $W(E_6)$-representations and mixed Hodge structures.
\end{theorem}

\begin{proof}
We will suppress the field $\CC$ for brevity.
In the quotient $\Pi^n_X U \to \Sym^n_X U$, the $\fS_n$ action on $Y$ and hence $H^*(Y)$ is trivial.
Thus by transfer, it is enough to restrict to the cases $\Pi^n$.

Since we noted the monodromy $\pi_1(Y) \from H^*(F)$ is trivial, the associated Serre spectral sequence converging to $H^*(Y \times_X Z; \QQ)$ has $E_2$ page
\[E_2^{p,q} \cong H^p(Y; \QQ) \tensor H^q(F; \QQ) \dispunct.\]
It is then enough to show that this spectral sequence degenerates immediately.
Since $H^*(F) = H^*(S^n) \cong H^*(S)^{\tensor n}$, and the differentials satisfy the Leibniz rule, we reduce to the $n = 1$ case.
Note that the differentials in these spectral sequences must be $W(E_6)$-equivariant, since the respective bundles are.

To describe $H^*(S)$ as a $W(E_6)$-representation, identify $S$ as the blowup of $\PP^2$ at $6$ points, with the exceptional divisors constituting $6$ of the $27$ lines.
It follows that $H^2(S)$ is generated by the canonical class of $S$ and the classes of these $6$ lines, and implies that as $W(E_6)$ representations,
\begin{equation}\label{surface-cohomology}
H^2(S; \QQ) \cong V_1 \dsum V_6 \dispunct.
\end{equation}
But there is no copy of the irreducible fundamental representation $V_6$ in $E^{2,1} = 0$ or $E^{3,0} \cong H^3(Y; \QQ)$.
Thus it remains to show that the differentials vanish on the $W(E_6)$-invariant classes in $H^*(S)$, or equivalently (by transfer) that the differentials in the Serre spectral sequence for the bundle $U \to X$ vanish.
The latter is shown in the proof of Corollary~1.5 in \cite{Das2019}.
\end{proof}

\begin{corollary}
For $Z$ and $F$ as above, and $G < W(E_6)$,
\[H^*((Y/G) \times_X Z) \cong (H^*(Y(\CC); \QQ) \tensor H^*(F; \QQ))^G \dispunct.\]
In particular, taking $G = W(E_6)$,
\[H^*(Z) \cong (H^*(Y(\CC); \QQ) \tensor H^*(F; \QQ))^{W(E_6)} = H^*(Y(\CC); \QQ) \tensor_{W(E_6)} H^*(F; \QQ) \dispunct.\]
\end{corollary}

Thus these computations reduce to elementary representation theory, namely character theory.
For convenience, a character table of $W(E_6)$ is reproduced in \cref{character-table}.

\subsection{The Grothendieck--Lefschetz trace formula and point counts}
\label{twisted-point-counts}
The varieties above are all smooth and quasiprojective.
Therefore their points counts over the finite field $\FF_q$ can be obtained from the action of $\Frob_q$ on their étale cohomology via the Grothendieck--Lefschetz trace formula:
\[\#Z(\FF_q) = q^{\dim Z} \sum_{i \ge 0} (-1)^i \Tr(\Frob_q \from H^{i}_{\text{ét}}(Z(\overline{\FF}_q); \QQ_\ell)^\vee) \dispunct.\]
This formula holds for a smooth $Z$ and a prime $\ell$ not dividing $q$.

The Cayley--Salmon theorem holds over any algebraically closed field, in particular $\overline{\FF}_q$.
The identification of the Galois group of $Y \to X$ with $W(E_6)$ implies that if $S$ is defined over $\FF_q$, then $\Frob_q$ permutes the $27$ lines by some element of $W(E_6) < \fS_{27}$.
We denote this by $\Frob_{q,S}$.
This also determines how $\Frob_q$ acts on $H^*_{\text{ét}}(S)$ as follows.

Recall that a smooth cubic surface $S$ is the blowup of $\PP^2$ at $6$ points and
\begin{equation}
H^2(S; \QQ) \cong V_1 \dsum V_6 \dispunct.
\end{equation}
Since $\Frob_q$ acts on $H^{2i}_{\text{ét}}(\PP^n(\overline{\FF}_q))$ by the scalar $q^{-i}$, we deduce that $\Frob_q$ acts on $H^0(S)$ as identity, on $H^2(S)$ by the action of $q^{-1} \cdot \Frob_{q,S} \in \QQ[W(E_6)]$ on $V_1 \dsum V_6$ and on $H^4(S)$ by the scalar $q^{-2}$.
This also proves \cref{point-count-by-point-count-corollary}.
It remains to prove \cref{point-count-theorem}.

\begin{remark}\label{fiber-constructions-determined-by-conjugacy}
The action of $\Frob_q$ on $H^*(S)$ determines the action of $\Frob_q$ on $H^*(S^n)$ and $H^*(\Sym^n)$.
Totaro \cite{Totaro1996} describes $H^*(\PConf^n S)$ as the cohomology of a DGA generated by $H^*(S^n)$ along with classes in degree $3$ which can be identified with the generator of $H^3(\CC^2 - 0)$.
This is enough to determine the action of $\Frob_q$ on $H^*(\PConf^n S)$ and, since this description is $\fS_n$ equivariant, on $H^*(\UConf^n S)$.
In particular, the conjugacy class of $\Frob_{q,S}$ determines the number of $\FF_q$ points on $S^n$, $\Sym^n S$, etc, again via the Grothendieck--Lefschetz trace formula.
\end{remark}

\begin{proof}[{Proof of \cref{point-count-theorem}}]
We want to count points of $X$ depending on the conjugacy class of $\Frob_q$ in $W(E_6)$.
Thus we will apply a version of the Grothendieck--Lefschetz trace formula with local coefficients (see e.g. \cite[\S3]{DL1976}), specifically for representations of $\pi_1(X)$ that factor through the finite group $W(E_6)$.
By transfer, this exactly corresponds to the cohomology of $Y$; more explicitly, for a $W(E_6)$-representation $V$,
\[H^*(X; V) \cong H^*(Y; \QQ) \tensor_{\QQ[W(E_6)]} V \dispunct.\]
To combine this with the Grothendieck--Lefschetz trace formula, we need knowledge of how $W(E_6)$ and $\Frob_q$ acts on $H_{\text{ét}}^*(Y)$.

Even though \cref{cohomology-of-Y} is stated for singular cohomology of the complex points, the same results hold (after tensoring with $\QQ_\ell$) for $H^*_{\text{ét}}(Y(\FF_q) ; \QQ_\ell)$ for every $q$.
Further, it is an important part of the results of Bergvall--Gounelas in \cite{BG2019} that $H^*(Y/\PGL(4))$ is \emph{minimally pure}, i.e.\ that $\Frob_q$ acts on $H^i_{\text{ét}}$ by the scalar $q^{-i}$.
In fact, the argument in \cite{BG2019} in some sense reverses the steps here, to go from point count to etalé cohomology to singular cohomology.

Equipped with this knowledge and using some linear algebra, one obtains
\[\# \set{S \in X}{\Frob_{q,S} \in c} = \#\PGL(4,\FF_q) \sum_{i \ge 0} (-1)^i q^{4-i} \inner{H^i}{\chi_c}_{W(E_6)} \]
where $\chi_c$ is the characteristic function of the conjugacy class $c$ and $H^i$ are as in \cref{cohomology-of-Y}.
This formula is enough to compute \cref{point-count-by-conjugacy-class} and also explains the normalizing factors in its second column.
\end{proof}

We can also state a point count version of \cref{cohomology-theorem}.
We already saw in \cref{fiber-constructions-determined-by-conjugacy} that $\# S^n(\FF_q)$, $\# \PConf^n S(\FF_q)$ etc only depend on $\Frob_{q,S}$.
More elementary is the fact that the cardinality of the fiber of $Y/G \to X$ over $S \in X(\FF_q)$ is determined by $\Frob_{q,S}$ for any $G < W(E_6)$, since this fiber is isomorphic to $W(E_6)/G$ as a $G$-set.
Thus we get the following statement whose proof is obvious.
\begin{proposition} \label{bundle-point-count-theorem}
Let $Z$ be $\Pi^n_X U$, $\Sym^n_X U$, $\PConf^n_X U$ or $\UConf^n_X U$ and $G < W(E_6)$.
Let $c$ be any conjugacy class of $W(E_6)$ and set $X_c = \set{S \in X(\FF_q)}{\Frob_{q,S} \in c}$.
For any $S \in X_c$, let $d = \# (Y/G)_S$ and let $F$ be the fiber of $Z \to X$ over $S$ (i.e.\ $S^n$, $\Sym^n S$, $\PConf^n S$ or $\UConf^n S$ respectively).
Then
\[\# \big[((Y \times_X Z)/G)(\FF_q) \times_{X(\FF_q)} X_c\big] = d \times (\# F) \times (\# X_c) \dispunct.\]
\end{proposition}

Since we know the distribution of $\Frob_{q,S}$, i.e. $\#X_c$ for each $c$, we can find the distribution of $d \times \#F$.
Two examples of such distributions are tabulated in \cref{tritangents-table,pairs-table}.
Specifically, For \cref{pairs-table}, $H^*(\UConf^2(S); \QQ)$ as a $W(E_6)$ representation can be computed to be the following using the spectral sequence in \cite{Totaro1996}:
\[H^0 \cong V_1\dispunct; \qquad H^2 \cong V_1 \dsum V_6 \dispunct; \qquad H^4 \cong V_1^{\dsum 2} \dsum V_6 \dsum V_{20} \dispunct; \qquad H^i = 0 \text{ for other } i \dispunct.\]


\begin{table}[p]
\centering
\caption{Distribution of tritangents over $S \in X(\FF_q)$. The average is $1$.}
\label{tritangents-table}
\begin{tabular}{RR}
\toprule
N & \dfrac{\#W(E_6)}{\#\PGL(4,\FF_q)} \times \#\set{S}{\text{$S$ has $N$ tritangents}}\\\midrule
0 & 576(38q^3 - 5q^2 + 5)q \\
1 & 540(39q^4 + 3q^3 + 3q^2 - 3q - 10) \\
2 & 2160(q^2 - q + 1)(q + 1)q \\
3 & 240(17q^4 - 25q + 24) \\
4 & 1440(q^2 + q - 1)(q + 1)q \\
5 & 270(q + 1)^2(q - 2)(q - 3) \\
6 & 240(q^2 - q + 1)(q + 1)q \\
7 & 540(q^2 - 3)(q - 1)q \\
9 & 80(q^2 + q - 3)(q + 1)^2 \\
13 & 45(q^2 - 2q - 7)(q - 2)(q - 3) \\
15 & 36(q^2 - 4q + 5)(q - 1)q \\
45 & (q - 2)(q - 3)(q - 5)^2 \\
\bottomrule
\end{tabular}
\end{table}

\begin{table}[p]
\centering
\caption{Distribution of $\#\UConf^2(S)$ over $S \in X(\FF_q)$. The average is $q^2(q^2+q+2)$.}
\label{pairs-table}
\begin{tabular}{LR}
	\toprule
	N                  & \dfrac{\#W(E_6)}{\#\PGL(4,\FF_q)} \times \#\set{S}{\#\UConf^2(S) = N} \\ \midrule
	q^4 - 2q^3 + q^2   &                                              80(q + 1)^2(q^2 + q - 3) \\
	q^4 - q^3 + q^2    &                                                        2880q(q^3 - 3) \\
	q^4 - q^3 + 2q^2   &                                                         540q(q - 1)^3 \\
	q^4 - q^3 + 4q^2   &                                        45(q - 2)(q - 3)(q^2 - 2q - 7) \\
	q^4 + q^2          &                                          864(q + 1)(11q^3 - 6q^2 + 5) \\
	q^4 + 2q^2         &                                             2160q(q + 1)(q^2 - q + 1) \\
	q^4 + q^3 + q^2    &                                            960(11q^4 - 6q^3 + 5q + 6) \\
	q^4 + q^3 + 2q^2   &                                          3240(q + 1)(3q - 2)(q^2 + 1) \\
	q^4 + q^3 + 4q^2   &                                                  540q(q - 1)(q^2 - 3) \\
	q^4 + 2q^3 + q^2   &                                       720(q + 1)(q^3 - 2q^2 + 2q - 3) \\
	q^4 + 2q^3 + 2q^2  &                                             4320q(q - 1)(q^2 + q + 1) \\
	q^4 + 2q^3 + 3q^2  &                                                      5184q^2(q^2 + 1) \\
	q^4 + 2q^3 + 4q^2  &                                                               2880q^4 \\
	q^4 + 3q^3 + 4q^2  &                                                   540(q + 1)^3(q - 2) \\
	q^4 + 3q^3 + 6q^2  &                                                 1620q(q + 1)^2(q - 1) \\
	q^4 + 3q^3 + 8q^2  &                                            270(q + 1)^2(q - 2)(q - 3) \\
	q^4 + 4q^3 + 10q^2 &                                              240q(q + 1)(q^2 - q + 1) \\
	q^4 + 5q^3 + 16q^2 &                                              36q(q - 1)(q^2 - 4q + 5) \\
	q^4 + 7q^3 + 28q^2 &                                               (q - 2)(q - 3)(q - 5)^2 \\ \bottomrule
\end{tabular}
\end{table}

\begin{remark}
	It is possible to deduce \cref{bundle-point-count-theorem} by applying the trace formula to \cref{cohomology-theorem} for $Z = \Pi^n_X U$ or $\Sym^n_X U$, and to the Serre spectral sequence of the bundle $Y \times_X Z \to Y$ for $Z = \PConf^n_X U$ or $Z = \UConf^n_X U$.
	The differentials in this spectral sequence are irrelevant for this, since the trace formula uses an alternating sum of traces, similar to the Euler characteristic.
	We leave the details to experts.
\end{remark}

\section{Character table of \texorpdfstring{$W(E_6)$}{W(E6)}}
\label{conjugacy-class-notation}
\label{character-table-section}
In this section we explain our notation for the conjugacy classes and irreducible representations, borrowing heavily from \cite{Frame1951}.
The conjugacy classes will be denoted by ``virtual cycle types'' as determined by their action on the $6$-dimensional \emph{fundamental representation} $V_6$.
In more detail, elements from different conjugacy classes in $W(E_6)$ remain non-conjugate in the representation $V_6$.
Thus a conjugacy class $c$ can be identified by the action of any $g \in c$ on $V_6$, and since this action of $\{1, g, g^2, \dots\} \cong \ZZ/n$ (i.e.\ $n$ is the order of $g$) is defined over $\QQ$, it can be decomposed as a (virtual) direct sum
\[\dSum_{d \divides n} \QQ[\ZZ/d]^{\dsum i_d} \dispunct.\]
Then we will denote the conjugacy class $c$ by the tuple $(d^{i_d})_{i_d \ne 0}$.
For example, the tuple
\[(1, 2^{-2}, 3^{-1}, 6^2)\]
denotes a conjugacy class of order $6$ whose elements act on $V_6$ with eigenvalues $(\zeta, \zeta, \zeta^2, \zeta^4, \zeta^5, \zeta^5)$, where $\zeta$ is a primitive $6$th root of unity.
We assure the reader that the tuples below do in fact represent actual (i.e.\ not just virtual) representations of dimension $6$, in particular they satisfy $\sum di_d = 6$.

The group $W(E_6)$ contains a simple subgroup $W(E_6)^+$ of index $2$.
Call the conjugacy classes in $W(E_6)$ that are contained in $W(E_6)^+$ \emph{even}, the others \emph{odd}.
In particular, $(1^4,2)$ is odd.
Every irreducible character of $W(E_6)$ that does not remain irreducible upon restriction to $W(E_6)^+$ vanishes on odd conjugacy classes.

There are 5 such irreducible representations, and they have different dimensions, so we will denote the one of dimension $n$ by $U_n$.
On the other hand, every irreducible character $\chi$ of $W(E_6)$ that \emph{does} remain irreducible when restricted satisfies $\chi(1^4,2) \ne 0$.
These occur in 10 pairs, differing in the sign of the character on odd conjugacy classes (and hence by a tensor product with the sign representation $V_1'$ pulled back from the non-trivial representation of the quotient $W(E_6)/W(E_6)^+$).
Each pair is denoted $V$ and $V'$ with subscripts, where $V$ has $\chi(1^4,2) > 0$, and $V'$ has $\chi(1^4,2) < 0$.
Distinct pairs have different dimensions, except two of the pairs have dimension $15$.
Hence we will denote them by $V_n$, $V_{n}'$ for $n = \dim \ne 15$, and $V_{15,1}$, $V_{15,1}'$, $V_{15,2}$, $V_{15,2}'$.
A full list of the $25$ irreducible representations and the notation for them in other sources is listed in \cref{irrep-table}.
\label{irrep-notation}

Finally, we include a character table of $W(E_6)$ as \cref{character-table}.
To avoid redundancy, we omit the characters of the representations denoted $V_n'$, these can be obtained by taking the character of the corresponding character $V_n$ and negating the values on the odd conjugacy classes.

\begin{landscape}
	\setlength{\tabcolsep}{3pt}
\begin{table}
\begin{minipage}{0.6\linewidth-5pt}
\centering
\caption{Conjugacy classes of $W(E_6)$}
\begin{tabular}{l@{\hskip -1ex}cccRRR}
\toprule
& \multicolumn{3}{c}{Notation for the class $c$} & \multicolumn{3}{c}{Properties of $c$ and $g \in c$} \\
&    This paper    & \cite{CCN+1985} & \cite{SwinnertonDyer1967} & \ord g & \# c & \# Z(g) \\
\midrule
\multirow{15}{*}{\rotatebox[origin=c]{90}{\emph{even classes: $g \in W(E_6)^+$}}}
&             $(1^6)$             &      1A      & $C_1$ &      1 &    1 &                               51840 \\
&          $(1^2, 2^2)$           &      2B      & $C_2$ &      2 &  270 &                                 192 \\
&         $(1^{-2},2^4)$          &      2A      & $C_3$ &      2 &   45 &                                1152 \\
&            $(1^3,3)$            &      3D      & $C_6$ &      3 &  240 &                                 216 \\
&         $(1^{-3},3^3)$          &    3A, 3B    & $C_{11}$ &      3 &   80 &                                 648 \\
&             $(3^2)$             &      3C      & $C_9$ &      3 &  480 &                                 108 \\
&       $(1^2,2^{-2},4^2)$        &      4A      & $C_4$ &      4 &  540 &                                  96 \\
&             $(2,4)$             &      4B      & $C_5$ &      4 & 3240 &                                  16 \\
&             $(1,5)$             &      5A      & $C_{15}$ &      5 & 5184 &                                  10 \\
&        $(1,2,3^{-1},6)$         &    6C, 6D    & $C_7$ &      6 & 1440 &                                  36 \\
&        $(1^{-1},2^2,3)$         &      6F      & $C_8$ &      6 & 2160 &                                  24 \\
&         $(1^{-2},2,6)$          &      6E      & $C_{10}$ &      6 & 1440 &                                  36 \\
&     $(1,2^{-2},3^{-1},6^2)$     &    6A, 6B    & $C_{12}$ &      6 &  720 &                                  72 \\
&          $(3^{-1},9)$           &    9A, 9B    & $C_{14}$ &      9 & 5760 &                                   9 \\
& $(1^{-1},2,3,4^{-1},6^{-1},12)$ &   12A, 12B   & $C_{13}$ &     12 & 4320 &                                  12 \\ \midrule
\multirow{10}{*}{\rotatebox[origin=c]{90}{\emph{odd classes: $g \notin W(E_6)^+$}}}
&            $(1^4,2)$            &      2C      & $C_{16}$ &      2 &   36 &                                1440 \\
&             $(2^3)$             &      2D      & $C_{17}$ &      2 &  540 &                                  96 \\
&            $(1^2,4)$            &      4D      & $C_{18}$ &      4 & 1620 &                                  32 \\
&        $(1^{-2},2^2,4)$         &      4C      & $C_{19}$ &      4 &  540 &                                  96 \\
&            $(1,2,3)$            &      6G      & $C_{21}$ &      6 & 1440 &                                  36 \\
&        $(1^{-2},2,3^2)$         &      6H      & $C_{22}$ &      6 & 1440 &                                  36 \\
&              $(6)$              &      6I      & $C_{23}$ &      6 & 4320 &                                  12 \\
&         $(2,4^{-1},8)$          &      8A      & $C_{20}$ &      8 & 6480 &                                   8 \\
&         $(1^{-1},2,5)$          &     10A      & $C_{25}$ &     10 & 5184 &                                  10 \\
&     $(1,2^{-1},3^{-1},4,6)$     &     12C      & $C_{24}$ &     12 & 4320 &                                  12 \\ \bottomrule&
\end{tabular}
\end{minipage} \hfill \begin{minipage}{0.4\linewidth-5pt}
\caption{Various notations for irreducible representations of $W(E_6)$}
\label{irrep-table}
\centering
\begin{tabular}{cccc}
	\toprule
	This paper  & \cite{Frame1951} &    \cite{CCN+1985}    & \cite{Carter1993} \\ \midrule
	   $V_1$    &      $1_p$       &       $\chi_1$        &   $\phi_{1,0}$    \\
	   $V_6$    &      $6_p$       &       $\chi_4$        &   $\phi_{6,1}$    \\
	$V_{15,1}$  &      $15_p$      &       $\chi_7$        &   $\phi_{15,5}$   \\
	$V_{15,2}$  &      $15_q$      &       $\chi_8$        &   $\phi_{15,4}$   \\
	 $V_{20}$   &      $20_p$      &       $\chi_9$        &   $\phi_{20,2}$   \\
	 $V_{24}$   &      $24_p$      &      $\chi_{10}$      &   $\phi_{24,6}$   \\
	 $V_{30}$   &      $30_p$      &      $\chi_{11}$      &   $\phi_{30,3}$   \\
	 $V_{60}$   &      $60_p$      &      $\chi_{18}$      &   $\phi_{60,5}$   \\
	 $V_{64}$   &      $64_p$      &      $\chi_{19}$      &   $\phi_{64,4}$   \\
	 $V_{81}$   &      $81_p$      &      $\chi_{20}$      &   $\phi_{81,6}$   \\ \midrule
	  $V_1'$    &      $1_n$       &                       &   $\phi_{1,36}$   \\
	  $V_6'$    &      $6_n$       &                       &   $\phi_{6,25}$   \\
	$V_{15,1}'$ &      $15_n$      &                       &  $\phi_{15,17}$   \\
	$V_{15,2}'$ &      $15_m$      &                       &  $\phi_{15,16}$   \\
	 $V_{20}'$  &      $20_n$      &                       &  $\phi_{20,20}$   \\
	 $V_{24}'$  &      $24_n$      &                       &  $\phi_{24,12}$   \\
	 $V_{30}'$  &      $30_n$      &                       &  $\phi_{30,15}$   \\
	 $V_{60}'$  &      $60_n$      &                       &  $\phi_{60,11}$   \\
	 $V_{64}'$  &      $64_n$      &                       &  $\phi_{64,13}$   \\
	 $V_{81}'$  &      $81_n$      &                       &  $\phi_{81,10}$   \\ \midrule
	 $U_{10}$   &      $10_s$      &    $\chi_2+\chi_3$    &   $\phi_{10,9}$   \\
	 $U_{20}$   &      $20_s$      &    $\chi_5+\chi_6$    &  $\phi_{20,10}$   \\
	 $U_{60}$   &      $60_s$      & $\chi_{12}+\chi_{13}$ &   $\phi_{60,8}$   \\
	 $U_{80}$   &      $80_s$      & $\chi_{14}+\chi_{15}$ &   $\phi_{80,7}$   \\
	 $U_{90}$   &      $90_s$      & $\chi_{16}+\chi_{17}$ &   $\phi_{90,8}$   \\ \bottomrule
\end{tabular}
\end{minipage}
\end{table}
\end{landscape}

\begin{table}
\caption{Character table of $W(E_6)$. See text for notation.}
\label{character-table}
\setlength{\tabcolsep}{3.5pt}
\begin{tabular}{l|RRRRRRRRRR|RRRRR}
\toprule
Class $c$                       & V_1 & V_6 & V_{15,1} & V_{15,2} & V_{20} & V_{24} & V_{30} & V_{60} & V_{64} & V_{81} & U_{10} & U_{20} & U_{60} & U_{80} & U_{90} \\ \midrule
$(1^6)$                         &   1 &   6 &       15 &       15 &     20 &     24 &     30 &     60 &     64 &     81 &     10 &     20 &     60 &     80 &     90 \\
$(1^2, 2^2)$                    &   1 &   2 &       -1 &        3 &      4 &      0 &      2 &      4 &      0 &     -3 &      2 &     -4 &      4 &      0 &     -6 \\
$(1^{-2},2^4)$                  &   1 &  -2 &       -1 &        7 &      4 &      8 &    -10 &     -4 &      0 &      9 &     -6 &      4 &     12 &    -16 &     -6 \\
$(1^3,3)$                       &   1 &   3 &        3 &        0 &      5 &      0 &      3 &     -3 &      4 &      0 &     -2 &      2 &     -6 &     -4 &      0 \\
$(1^{-3},3^3)$                  &   1 &  -3 &        6 &       -3 &      2 &      6 &      3 &      6 &     -8 &      0 &      1 &     -7 &     -3 &    -10 &      9 \\
$(3^2)$                         &   1 &   0 &        0 &        3 &     -1 &      3 &      3 &     -3 &     -2 &      0 &      4 &      2 &      0 &      2 &      0 \\
$(1^2,2^{-2},4^2)$              &   1 &   2 &        3 &       -1 &      0 &      0 &     -2 &      0 &      0 &     -3 &      2 &      4 &      4 &      0 &      2 \\
$(2,4)$                         &   1 &   0 &       -1 &        1 &      0 &      0 &      0 &      0 &      0 &     -1 &     -2 &      0 &      0 &      0 &      2 \\
$(1,5)$                         &   1 &   1 &        0 &        0 &      0 &     -1 &      0 &      0 &     -1 &      1 &      0 &      0 &      0 &      0 &      0 \\
$(1,2,3^{-1},6)$                &   1 &   1 &       -1 &       -2 &      1 &      2 &     -1 &     -1 &      0 &      0 &      0 &     -2 &      0 &      2 &      0 \\
$(1^{-1},2^2,3)$                &   1 &  -1 &       -1 &        0 &      1 &      0 &     -1 &      1 &      0 &      0 &      2 &      2 &     -2 &      0 &      0 \\
$(1^{-2},2,6)$                  &   1 &  -2 &        2 &        1 &      1 &     -1 &     -1 &     -1 &      0 &      0 &      0 &     -2 &      0 &      2 &      0 \\
$(1,2^{-2},3^{-1},6^2)$         &   1 &   1 &        2 &        1 &     -2 &      2 &     -1 &      2 &      0 &      0 &     -3 &      1 &     -3 &      2 &     -3 \\
$(3^{-1},9)$                    &   1 &   0 &        0 &        0 &     -1 &      0 &      0 &      0 &      1 &      0 &      1 &     -1 &      0 &     -1 &      0 \\
$(1^{-1},2,3,4^{-1},6^{-1},12)$ &   1 &  -1 &        0 &       -1 &      0 &      0 &      1 &      0 &      0 &      0 &     -1 &      1 &      1 &      0 &     -1 \\ \midrule
$(1^4,2)$                       &   1 &   4 &        5 &        5 &     10 &      4 &     10 &     10 &     16 &      9 &        &  \\
$(2^3)$                         &   1 &   0 &       -3 &        1 &      2 &      4 &     -2 &      2 &      0 &     -3 &        &  \\
$(1^2,4)$                       &   1 &   2 &        1 &       -1 &      2 &      0 &      0 &     -2 &      0 &     -1 &        &  \\
$(1^{-2},2^2,4)$                &   1 &  -2 &        1 &        3 &      2 &      0 &     -4 &     -2 &      0 &      3 &        &  \\
$(1,2,3)$                       &   1 &   1 &       -1 &        2 &      1 &     -2 &      1 &      1 &     -2 &      0 &        &  \\
$(1^{-2},2,3^2)$                &   1 &  -2 &        2 &       -1 &      1 &      1 &      1 &      1 &     -2 &      0 &        &  \\
$(6)$                           &   1 &   0 &        0 &        1 &     -1 &      1 &      1 &     -1 &      0 &      0 &        &  \\
$(2,4^{-1},8)$                  &   1 &   0 &       -1 &       -1 &      0 &      0 &      0 &      0 &      0 &      1 &        &  \\
$(1^{-1},2,5)$                  &   1 &  -1 &        0 &        0 &      0 &     -1 &      0 &      0 &      1 &     -1 &        &  \\
$(1,2^{-1},3^{-1},4,6)$         &   1 &   1 &        1 &        0 &     -1 &      0 &     -1 &      1 &      0 &      0 &        &  \\ \bottomrule
\end{tabular}
\end{table}

\clearpage

\printbibliography

\end{document}